\documentclass[a4paper,12pt]{article}
\usepackage{amsfonts}
\usepackage{amssymb}
\usepackage{amsmath}
\usepackage{amsthm}
\usepackage{bm}
\usepackage{here}
\usepackage[dvipdfmx]{graphicx}
\usepackage{tikz}
\usetikzlibrary{intersections,calc,arrows}
\usepackage{color}
\usepackage{cases}
\numberwithin{equation}{section}

\theoremstyle{definition}
\newtheorem{defn}{Definition}[section]
\newtheorem{rem}[defn]{Remark}
\theoremstyle{plain}
\newtheorem{lemm}[defn]{Lemma}

\newtheorem{thm}[defn]{Theorem}
\newtheorem{cor}[defn]{Corollary}

\title{An integral representation of eigenfunctions for the deformed Noumi--Sano operators}
\author{
	Taikei Fujii 
	and  
	Takahiko Nobukawa
}
\date{}
\allowdisplaybreaks
\begin{document}
\maketitle
\renewcommand{\labelenumi}{\rm (\arabic{enumi})}
\begin{abstract}	
The deformed Noumi--Sano operator $H_{n, m}^d(x, y)$ of order $d$ is a simultaneous extension of the Macdonald operator and the Noumi--Sano operator.
In this paper, we construct an  integral transform associated with eigenvalue problem of  $H_{n, m}^d(x, y)$. 
Using the integral transform, we obtain an integral representation of eigenfunctions for $H_{n, m}^d(x, y)$.
	
Key words: $q$-hypergeometric integral, Macdonald polynomial, $q$-difference equation.
	
2020 Mathematics Subject Classification Numbers: 39A13, 33D52, 44A20.
\end{abstract}

\section{Introduction}\label{secintro} 
We fix $q$, $t \in \mathbb{C}$ with $0 < |q|,\, |t| < 1$.
The shift operator $T_{q,x}^{\pm\mu}=\prod_{i=1}^n T_{q,x_i}^{\pm \mu_i}$ is defined by $(T_{q,x}^{\pm\mu}f)(x_1,\ldots,x_n)=f(q^{\pm \mu_1}x_1,\ldots,q^{\pm \mu_n}x_n)$.
Halln\"as, Langmann, Noumi and Rosengren \cite{HLNR} introduced a family of $q$-difference operators defined by
\begin{align}\label{DNSop}
		&\notag H_{n,m}(x,y;u; q, t) = H_{n,m}(x,y;u)= \sum_{d=0}^{\infty}u^d H_{n,m}^d(x,y)\\
		&:=\sum_{\mu\in(\mathbb{Z}_{\geq0})^n}\sum_{I\subset\{1,\ldots,m\}}(t^{1-n}q^m u)^{|\mu|}(-tu)^{|I|} q^{\binom{|I|}{2}}B_{\mu,I}(x,y)T_{q,x}^{\mu}T_{t,y}^{-I}
		,\\
		&\notag B_{\mu,I}(x,y)=\frac{\Delta(q^\mu x)}{\Delta(x)}\prod_{i,j=1}^n\frac{(tx_i/x_j;q)_{\mu_i}}{(qx_i/x_j;q)_{\mu_i}}\prod_{\substack{1\leq i,j\leq m\\i\in I,j\notin I}}\frac{y_i-qy_j}{y_i-y_j}\\
		&\quad \quad \quad \quad \quad \quad \times \prod_{i=1}^n \left(\prod_{j\in I}\frac{1-x_i/(ty_j)}{1-q^{\mu_i}x_i/y_j}\prod_{j\notin I}\frac{1-x_i/(q y_j)}{1-q^{\mu_i-1}x_i/y_j}\right),
	\end{align}
where $|\mu|=\mu_1+\cdots+\mu_n$, $q^\mu x=(q^{\mu_1}x_1,\ldots,q^{\mu_n}x_n)$, $\Delta(x)=\prod_{i<j}(x_i-x_j)$ and $(a; q)_{n} = \prod_{k=0}^{n - 1} (1 - a q^k)$. 
The operator $H_{n, m}(x, y;u)$ is called the deformed Noumi--Sano operator.
This operator $H_{n, m}(x, y;u)$ is a simultaneous extension of the Macdonald operator of type $A$ \cite{Mac, N} and the Noumi--Sano operator \cite{NS21}.
More precisely, $H_{0,m}(y^{-1}, u/t; t, q)$  coincides with the Macdonald operator, and $H_{n, 0}(x; u/t; q, t)$ is the Noumi--Sano operator (see \cite{FS, HLNR, SV09} for detailed properties).
The operator \eqref{DNSop} has super-Macdonald polynomials \cite{SV09} as eigenfunctions. 
Our aim is to construct an integral transform associated with the eigenvalue problem for the deformed Noumi--Sano operator $ H_{n,m}^d(x,y)$.

We define the function $\omega_{n,m;N,M}(x,y;z,w)$ of the variables $x = (x_1, \ldots, x_n)$, $y = (y_1, \ldots, y_m)$, $z = (z_1, \ldots, z_N)$ and $w = (w_1, \ldots, w_M)$ as
\begin{align}
	\omega_{n,m;N,M}(x,y;z,w)=x^{\alpha} y^{\beta}\Psi_{n,m;N,M}(x,y;z,w) \rho(x, y),
\end{align}
for $\alpha, \beta \in \mathbb{C}$, where $(a; p)_{\infty} = \prod_{k=0}^{\infty}(1 - a p^k)$ and
\begin{align}
&\rho(x,y) =
 \frac{
	\left(
	\prod_{1 \leq i \neq j \leq n}(x_i/x_j; q)_{\infty}/(t x_i/x_j; q)_{\infty}
	\right)
	\left(
	\prod_{1 \leq i \neq j \leq m}(y_i/y_j; t)_{\infty}/(q y_i/y_j; t)_{\infty}
	\right)
}{\prod_{i=1}^{n}\prod_{j=1}^{m} (1- q^{-1}x_i/y_j)(1 - t y_j/x_i)}, \\
&\notag \Psi_{n,m;N,M}(x,y;z,w)\\
&\label{KFforDNS}=
\prod_{i=1}^n\prod_{j=1}^N\frac{(x_iz_j;q)_\infty}{(x_iz_j/t;q)_\infty}\prod_{i=1}^m\prod_{j=1}^M \frac{(q t y_iw_j;t)_\infty}{(t y_iw_j;t)_\infty}\prod_{i=1}^n\prod_{j=1}^M (1- x_iw_j)\prod_{i=1}^m\prod_{j=1}^N(1- y_iz_j).
\end{align}
The main result of this paper is as follows.
\begin{thm}\label{Main thm}
Assume $q^{\alpha} = t^{-\beta}$, $r < |q|^d R$, and $|q| < |t|$, where $\alpha, \beta \in \mathbb{Z}$ and $r, R \in \mathbb{R}_{> 0}$. 
Suppose that a function  $\varphi(x, y)$ satisfies the following conditions: 
\begin{itemize}
\item[1.] The function $\varphi(x, y)$ is a simultaneous eigenfunction of $H^{k}_{n, m}$ $(0 \leq k \leq d)$, i.e.
\begin{align}
	H^{k}_{n,m}(t x^{-1},q^{-1}y^{-1}) \varphi(x, y) = E_k \varphi(x, y), \quad 0 \leq k \leq d.
\end{align}
\item[2.]
The function $\varphi(x, y)$ is holomorphic on $D = \{(x, y)\mid |q|^{d}R \leq |x_{i}| \leq R,\, r \leq |y_{j}| \leq |t|^{-1} r \, (1 \leq i \leq n, \, 1 \leq j \leq m) \}$.
\end{itemize}
We put $C_d = \mathbb{T}^n_{R} \times \mathbb{T}^m_{r}$, 
where $\mathbb{T}^s_{p}$ denotes the $s$-dimensional torus of radius $p$.
If $|z_{j}|<|t|/R , \, |w_{j}|<1/r$, then the function 
\begin{align}
\Omega(z, w)= \int_{C_d}  \varphi(x, y) \omega_{n,m;N,M}(x,y;z,w)  \frac{dx}{x}\frac{dy}{y}
\end{align}
satisfies 
\begin{align}
	H^{d}_{N,M}(z,w)\Omega(z, w)
	=
	\left( 
	\sum_{k+l = d}\frac{(t^{N - n} q^{m - M}; q)_k}{(q; q)_k} (t^{1-N}q^M)^k t^{\beta l}E_l
	\right) \Omega(z, w). 
	\end{align}
\end{thm}

\begin{rem}
The function $\rho(x, y)$ coincides with the function $\Delta_{n, m}(x, y; q, t)$ in \cite{AHL21} with $x$ and $y$ replaced by $t^{1/2}x$ and $q^{-1/2}y$, respectively.
The function $\Psi_{n,m;N,M}(x,y;z,w)$ is a kernel function of the deformed Noumi--Sano operator \cite[Theorem 2.1, Remark 2.2]{HLNR} as
 \begin{align}
&\notag H_{N, M}(z, w; u)\Psi_{n,m;N,M}(x,y;z,w)\\
 & = \frac{(t^{1-n}q^m u; q)_{\infty}}{(t^{1-N}q^M u; q)_{\infty}}H_{n, m}(x, y; u)\Psi_{n,m;N,M}(x,y;z,w). 
 \end{align} 
 \end{rem}
 \begin{rem}
 When $n=N=0$, our integral $\Omega(z, w)$ coincides with the integral in \cite{MN97}.
\end{rem}
Due to the explicit formula \cite[equation (36)]{HLNR}
\begin{align}
	H^{l}_{n,m}(x, y) 1 = \frac{(t^{n}q^{-m};q)_{l}}{(q;q)_{l}} (t^{1-n}q^{m})^l 1, 
\end{align}
we obtain an eigenfunction of $H^{d}_{N,M}(z,w)$ as follows:
\begin{cor}
Assume $q^{\alpha} = t^{-\beta}$, $r < |q|^d R$, and $|q| < |t|$, where $\alpha, \beta \in \mathbb{Z}$ and $r, R \in \mathbb{R}_{> 0}$. 
If $|z_{j}|<|t|/R , \, |w_{j}|<1/r$, then the following equation holds: 
\begin{align}
	&H^{d}_{N,M}(z,w)\int_{\mathbb{T}^n_{R} \times \mathbb{T}^m_{r}} \omega_{n,m;N,M}(x,y;z,w)  \frac{dx}{x}\frac{dy}{y} \notag \\
	&\notag=
	\left( 
	\sum_{k+l = d}\frac{(t^{N - n} q^{m - M}; q)_k (t^{n}q^{-m};q)_{l}}{(q; q)_k (q;q)_{l}} (t^{1-N}q^M)^k (q^{m }t^{\beta + 1- n})^l
	\right) \\
	&\quad \quad \quad \quad \times
	\int_{\mathbb{T}^n_{R} \times \mathbb{T}^m_{r}} \omega_{n,m;N,M}(x,y;z,w)  \frac{dx}{x}\frac{dy}{y}. 
	\end{align}
\end{cor}

The contents of this paper are as follows. 
In Section \ref{RevieMN}, we review \cite{MN97}. 
In Section \ref{Proof of Main Theorem}, we give a proof of Theorem \ref{Main thm}.
In Section \ref{secsum}, we summarize the results and discuss related problems.
In Appendix \ref{Appendix}, we construct an integral transform for the eigenvalue problem of the Noumi--Sano operator $H_{n, 0}(x;u)$.
\section{Review of Mimachi--Noumi's method}\label{RevieMN}
In this paper, we construct an integral transform for deformed Noumi--Sano operators using the method of Mimachi--Noumi \cite{MN97}.
In this section, we give a review of this method. 

Macdonald introduced a commuting family of $q$-difference operators defined by 
\begin{align}\label{MRop}
D_y^r = t^{\binom{r}{2}} \sum_{\substack{I \subset \{1,2, \ldots, m\}\\ |I| = r}} \prod_{i \in I, j \notin I} \frac{t y_i - y_j}{y_i - y_j} T_{q, y}^I, 
\end{align}
for $y = (y_1, \ldots, y_m)$ and $r = 1, 2, \ldots, m$. 
The generating function of these operators is given by 
\begin{align}\label{geneMRop}
D_y(u) = \sum_{r=0}^m(-u)^r D_y^r
\end{align}
The operator $D_y(u)$ satisfies the following important kernel identity.
\begin{thm}[Macdonald \cite{Mac}, Mimachi--Noumi \cite{MN97}, Noumi \cite{N}]\label{KIDforMRop}
We set 
\begin{align}
\Pi_{m,M}(y;w) = \prod_{i=1}^m \prod_{j=1}^M \frac{(t y_i w_j;q)_\infty}{(y_i w_j;q)_\infty}.
\end{align}
Then we have
\begin{align}\label{KIDforMRop_eq}
D_y(u)\Pi_{m,M}(y;w) = (u; t)_{m - M} D_w(t^{m - M}u)\Pi_{m,M}(y;w).
\end{align}
\end{thm}
\begin{proof}
The identity \eqref{KIDforMRop_eq} is independent of $q$. 
Hence it suffices to prove the case $q = t$.
When $q = t$, we obtain the following from Cauchy formula.
\begin{align}\label{proofofKIDforMRop_Pi}
\Pi_{m,M}(y;w) = \prod_{i=1}^m \prod_{j=1}^M \frac{1}{1 - y_i w_j} = \sum_{l (\lambda)\leq \min \{m, M\}} s_{\lambda}(y)s_{\lambda}(w).
\end{align}
Suppose $M \geq m$, we have 
\begin{align}\label{proofofKIDforMRop_s}
&D_y(u)s_\lambda(y) = s_\lambda(y) \prod_{i=1}^m (1 - ut^{M - i + \lambda_i}) \prod_{i=m+1}^M (1 -u t^{M - i}),\notag \\
&D_w(u)s_\lambda(w) = s_\lambda(w) \prod_{i=1}^m (1 -u t^{M - i + \lambda_i}).
\end{align}
By \eqref{proofofKIDforMRop_Pi} and \eqref{proofofKIDforMRop_s}, we obtain the special case $q = t$ of the desired formula. 
This completes the proof. 
\end{proof}
The proof of our main results is based on a method similar to that of Mimachi--Noumi \cite{MN97}.
Therefore, we review Mimachi--Noumi's result and their proof.  
\begin{thm}[Mimachi--Noumi \cite{MN97}]\label{RevieMN_mainthm}
Assume $\alpha \in \mathbb{Z}$. 
Suppose that a function $\varphi(y)$ satisfies the following conditions: 
\begin{itemize}
\item[1.] The function $\varphi(y)$ is an eigenfunction of $D_{y^{-1}}(u/q^\alpha)$, i.e.
\begin{align}\label{assumeeigenMR}
	D_{y^{-1}}(u/q^{\alpha}) \varphi(y) = E \varphi(y).
\end{align}
\item[2.]
The function $\varphi(y)$ is holomorphic on $V = \{y\mid |q| \leq |y_{i}| \leq 1 \  (1 \leq i \leq m) \}$.
\end{itemize}
If $|w_j|<1$, then the function
\begin{align}
\Omega(w) =  \int_{\mathbb{T}^m}  \varphi(y) \omega_{m, M}(y; w)  \frac{dy}{y}
\end{align}
satisfies
\begin{align}\label{Int_trans_MR}
D_w(u) \Omega(w) = (u;t)_{M-m} E \Omega(w), 
\end{align}
where 
\begin{align}
\omega_{m, M}(y; w) =y^\alpha \Pi_{m,M} (y;w) \rho(y),\quad \rho(y) = \prod_{i \neq j}\frac{(y_i/y_j;q)_{\infty}}{(t y_i/y_j;q)_{\infty}}.
\end{align}
\end{thm}

\begin{rem}
In \cite[Theorem 1]{MN97}, another integral transform in terms of Jackson integral is constructed. 
Here we state only the part relevant to the present paper. 
\end{rem}

The proof of Theorem \ref{RevieMN_mainthm} will be given after we prove several lemmas. 
These lemmas are obtained in \cite{MN97}.

\begin{lemm}\label{adforMRop}
We put 
\begin{align}
D_y(u) = \sum_{I \subset \{ 1, \ldots, m \}} (-u)^{|I|} A_I(y) T_{q,y}^I, \quad 
A_I(y) = t^{r(r-1)/2} \prod_{i \in I, j \notin I} \frac{t y_i - y_j}{y_i - y_j}.
\end{align}
Then we have
\begin{align}
T_{q,y}^{-I}A_I(y) \rho(y) = A_I(y^{-1})\rho(y).
\end{align}
\end{lemm}

\begin{proof}
By a simple calculation, we have 
\begin{align}
\frac{T_{q,y}^{-I}\rho(y)}{\rho(y)}
=
\prod_{i \in I, j \notin I} 
\frac{1 - y_i/(qy_j)}{1 - t y_i/(qy_j)}
\frac{1 -t y_j/y_i}{1 -  y_j/y_i}.
\end{align}
Hence we obtain
\begin{align}
\frac{T_{q,y}^{-I}A_I(y) \rho(y)}{\rho(y)}
&\notag=
\prod_{i \in I, j \notin I} 
\frac{1 - y_i/(qy_j)}{1 - t y_i/(qy_j)}
\frac{1 - t y_j/y_i}{1 -  y_j/y_i}
\prod_{i \in I, j \notin I} \frac{t y_i/q - y_j}{y_i/q - y_j}\\
&=
\prod_{i \in I, j \notin I} 
\frac{1 -t y_j/y_i}{1 -  y_j/y_i}
=A_I(y^{-1}).
\end{align}
\end{proof}

\begin{lemm}\label{qde-RhamforMR}
Assume $\alpha \in \mathbb{Z}$ and $|t|<|q|$.
Suppose a function $\varphi(y)$ is holomorphic on $V = \{y\mid |q| \leq |y_{i}| \leq 1 \ (1 \leq i \leq m) \}$.
If $|w_i|<1$, then
\begin{align}\label{qde-RhamforMR_eq}
&\notag \int_{\mathbb{T}^m}
  \varphi(y)
  y^{\alpha}
  A_{I}(y)
 \left(
T_{q,y}^I\Pi(y;w)
 \right)
 \rho(y)
  \frac{dy}{y}\\
&  =
  q^{-\alpha |I|}
 \int_{\mathbb{T}^m}
  y^{\alpha}
  \Pi(y;w)
 \rho(y)
  A_{I}(y^{-1})
  \left(
   T_{q,y}^{-I}   \varphi(y)
    \right)
  \frac{dy}{y}.
\end{align}
\end{lemm}
\begin{proof}
By the change of variables $y\to q^{-I}y$, we have
\begin{align}
&\notag\int_{\mathbb{T}^m}
 \varphi(y)
  y^{\alpha}
  A_{I}(y)
 \left(
T_{q,y}^I\Pi(y;w)
 \right)
 \rho(y)
  \frac{dy}{y}
  \\
  &=
  q^{-\alpha |I|}
 \int_{q^{I} \mathbb{T}^m}
  y^{\alpha}
  \left(
  T_{q,y}^{-I}  A_{I}(y) \rho(y) \varphi(y)
    \right)
    \Pi(y;w)
  \frac{dy}{y}.
\end{align}
By Lemma \ref{adforMRop}, we get 
\begin{align}\label{Lem_MR_deRham_proof_eq}
&\notag
\int_{q^{I} \mathbb{T}^m}
  y^{\alpha}
  \left(
  T_{q,y}^{-I}  A_{I}(y) \rho(y) \varphi(y)
    \right)
    \Pi(y;w)
  \frac{dy}{y}
\\=
&
 \int_{q^{I} \mathbb{T}^m}
  y^{\alpha}
  \left(
  T_{q,y}^{-I}  \varphi(y)
    \right)
    A_I(y^{-1})\rho(y)
    \Pi(y;w)
  \frac{dy}{y}.
\end{align}
The poles of the integrand \eqref{Lem_MR_deRham_proof_eq} on the right-hand side for $y \in (\mathbb{C}\setminus \{0\})^m$ are as follows: 
\begin{align}
&\notag y_i - q^{-k}/w_j = 0 \quad \left(k \in \mathbb{Z}_{\geq 0},\ 1 \leq i \leq m,\ 1 \leq j \leq M \right), \\
&y_a - t q^{k} y_b = 0 \quad \left(k \in \mathbb{Z}_{\geq 0},\ 1 \leq a, b \leq m \right). 
\end{align}
From assumptions $|w_{j}|<1$ and $|t| < |q|$, the integrand \eqref{Lem_MR_deRham_proof_eq} on the right-hand side is holomorphic on $V$.
Therefore, we obtain
\begin{align}
&\notag \int_{q^{I} \mathbb{T}^m}
  y^{\alpha}
  \left(
  T_{q,y}^{-I}  \varphi(y)
    \right)
    A_I(y^{-1})\rho(y)
    \Pi(y;w)
  \frac{dy}{y}\\
&  =\int_{\mathbb{T}^m}
  y^{\alpha}
  \left(
  T_{q,y}^{-I}  \varphi(y)
    \right)
    A_I(y^{-1})\rho(y)
    \Pi(y;w)
  \frac{dy}{y}
\end{align}
From this equation, we obtain the desired equation \eqref{qde-RhamforMR_eq}.
\end{proof}

We give the proof of Theorem \ref{RevieMN_mainthm}, following \cite{MN97}.
\begin{proof}[Proof of Theorem \ref{RevieMN_mainthm}.]
By the kernel identity \eqref{KIDforMRop_eq}, we get
\begin{align}
D_w(u) \Omega(w) 
&\notag=
 \int_{\mathbb{T}^m}
  \varphi(y)
  y^{\alpha}
 \left(
 D_w(u)\Pi(y;w)
 \right)
 \rho(y)
  \frac{dy}{y}
  \\
  &=
 (u;t)_{M-m} 
 \int_{\mathbb{T}^m}
  \varphi(y)
  y^{\alpha}
 \left(
 D_y(u)\Pi(y;w)
 \right)
 \rho(y)
  \frac{dy}{y}.
 \end{align}
 By Lemma \ref{qde-RhamforMR}, we have 
 \begin{align}
 &\notag
 \int_{\mathbb{T}^m}
  \varphi(y)
  y^{\alpha}
 \left(
 D_y(u)\Pi(y;w)
 \right)
 \rho(y)
  \frac{dy}{y}\\
&\notag  =
  \sum_{I \subset \{1, \ldots, m \}}
 \int_{\mathbb{T}^m}(-u)^{|I|}
  \varphi(y)
  y^{\alpha}
  A_{I}(y)
 \left(
T_{q,y}^I\Pi(y;w)
 \right)
 \rho(y)
  \frac{dy}{y}
  \\
  &=
   \sum_{I \subset \{1, \ldots, m \}}
 \int_{\mathbb{T}^m}
  \varphi(y)
  y^{\alpha}
  \Pi(y;w)
 \rho(y)
 (u/q^{\alpha})^{|I|}
  A_{I}(y^{-1})
  \left(
   T_{q,y}^{-I}   \varphi(y)
    \right)
  \frac{dy}{y}.
 \end{align}
 By assumption \eqref{assumeeigenMR}, we get
 \begin{align}
\notag
& \sum_{I \subset \{1, \ldots, m \}}
 \int_{\mathbb{T}^m}
  \varphi(y)
  y^{\alpha}
  \Pi(y;w)
 \rho(y)
 (u/q^{\alpha})^{|I|}
  A_{I}(y^{-1})
  \left(
   T_{q,y}^{-I}   \varphi(y)
    \right)
  \frac{dy}{y}
\\
& =
\int_{\mathbb{T}^m}
\left(
D_{y^{-1}}(u/q^{\alpha}) \varphi(y)
\right)
\omega_{m, M}(y; w) 
 \frac{dy}{y}
 =E
 \int_{\mathbb{T}^m}
 \varphi(y)
\omega_{m, M}(y; w) 
 \frac{dy}{y}.
\end{align}
Hence we have the desired equation \eqref{Int_trans_MR}.
\end{proof}
\section{Integral transform for eigenproblem of deformed Noumi--Sano operator}\label{Proof of Main Theorem}
In this section, we construct an integral transform for deformed Noumi--Sano operators using the method of Mimachi--Noumi \cite{MN97}.

The deformed Noumi--Sano operator $H_{n, m}(x, y; u)$ satisfies the following kernel identity. 
\begin{thm}[Halln\"as--Langmann--Noumi--Rosengren \cite{HLNR}]\label{KIDforDNS}
The following identity holds. 
 \begin{align}\label{KIDforDNSeq}
&\notag H_{N, M}(z, w; u)\Psi_{n,m;N,M}(x,y;z,w) \\
&= \frac{(t^{1-n}q^m u; q)_{\infty}}{(t^{1-N}q^M u; q)_{\infty}}H_{n, m}(x, y; u)\Psi_{n,m;N,M}(x,y;z,w).
 \end{align} 
 Here $\Psi_{n,m;N,M}(x,y;z,w)$ is given in \eqref{KFforDNS}.
 In particular, we have
 \begin{align}\label{KID rank d}
& H_{N, M}^d(z, w)\Psi_{n,m;N,M}(x,y;z,w) \notag \\
&=
\sum_{k+l = d}\frac{(t^{N - n} q^{m - M}; q)_k}{(q; q)_k} (t^{1-N}q^M)^k H_{n, m}^l(x, y)\Psi_{n,m;N,M}(x,y;z,w).
\end{align} 
 \end{thm}
The identity \eqref{KIDforDNSeq} follows from Kajihara's transformation formula \cite{K04}.
We omit the proof.
For details, see \cite{HLNR}.
We note that $\Psi$ in \cite{HLNR} is defined as 
\begin{align}
\prod_{i=1}^n\prod_{j=1}^N\frac{(t x_iz_j;q)_\infty}{(x_iz_j;q)_\infty}\prod_{i=1}^m\prod_{j=1}^M \frac{(q t y_iw_j;t)_\infty}{(t y_iw_j;t)_\infty}\prod_{i=1}^n\prod_{j=1}^M (1- t x_iw_j)\prod_{i=1}^m\prod_{j=1}^N(1- t y_iz_j).
\end{align}
However, this $\Psi$ does not satisfy the kernel identity \eqref{KIDforDNSeq}.
We slightly modify $\Psi$ so that it satisfies the kernel identity.

To prove Theorem \ref{Main thm}, we prove several lemmas.
\begin{lemm}\label{Lemm adjoint of B}
	\begin{align}\label{adjoint of B}
		T_{q,x}^{-\mu}T_{t,y}^{I} (B_{\mu,I}(x,y)\rho(x, y))=B_{\mu,I}(t x^{-1},q^{-1}y^{-1})\rho(x, y).
	\end{align}

\end{lemm}
\begin{proof}
We put
\begin{align}
&F_0(x, y)=\prod_{i=1}^n \left(\prod_{j\in I}\frac{1-x_i/(ty_j)}{1-q^{\mu_i}x_i/y_j}\prod_{j\notin I}\frac{1-x_i/(q y_j)}{1-q^{\mu_i-1}x_i/y_j}\right), \\
&G_0(x, y)=\prod_{i=1}^{n}\prod_{j=1}^{m} \frac{1}{(1- q^{-1}x_i/y_j)(1 - t y_j/x_i)}, \\
&F_1(x)=\frac{\Delta(q^\mu x)}{\Delta(x)}\prod_{i,j=1}^n\frac{(tx_i/x_j;q)_{\mu_i}}{(qx_i/x_j;q)_{\mu_i}}, \\
&G_1(x)=\prod_{1 \leq i \neq j \leq n}\frac{(x_i/x_j; q)_{\infty}}{(t x_i/x_j; q)_{\infty}}, \\
&F_2(y)=\prod_{\substack{1\leq i,j\leq m\\i\in I,j\notin I}}\frac{y_i-qy_j}{y_i-y_j}, \\
&G_2(y)=\prod_{1 \leq i \neq j \leq m}\frac{(y_i/y_j; t)_{\infty}}{(q y_i/y_j; t)_{\infty}}
.
\end{align}
By simple calculation, we have 
\begin{align}
&T_{q,x}^{-\mu}T_{t,y}^{I} (F_0(x, y) G_0(x, y)) = F_0(t/x, 1/(q y)) G_0(x, y), \\
&T_{q,x}^{-\mu}(F_1(x) G_1(x)) = F_1(t/x) G_1(x), \\
&T_{t,y}^{I} (F_2(y) G_2(y)) = F_2(1/(qy)) G_2(y).  
\end{align}
Since $B_{\mu,I}(x,y)\rho(x, y) = F_0(x, y) G_0(x, y)F_1(x) G_1(x)F_2(y) G_2(y)$, we obtain the desired identity \eqref{adjoint of B}. 
\end{proof}
 \begin{rem}
The case $|I|=|\mu|=1$ of Lemma \ref{Lemm adjoint of B} was obtained in \cite{AHL21}.
\end{rem}

\begin{lemm}\label{int to int term mu I}
Assume $r < |q|^{\max_{i} {|\mu_i|}} R$ and $|q| < |t|$, where $r, R \in \mathbb{R}_{> 0}$.
Suppose a function $\varphi(x, y)$ is holomorphic on $D = \{(x, y)\mid |q|^{\mu_i}R \leq |x_{i}| \leq R,\, r \leq |y_{j}| \leq |t|^{-1} r \, (1 \leq i \leq n, \, 1 \leq j \leq m) \}$. 
If $|z_{j}|<|t|/R , \, |w_{j}|<1/r$, then we have
	\begin{align}\label{int to int term mu I eq}
	&\int_{C_\mu} x^{\alpha} y^{\beta} \varphi (x, y) \rho(x,y) B_{\mu, I}(x, y) 
	\left(
	T_{q, x}^{\mu} T_{t, y}^{-I} \Psi_{n,m;N,M}(x,y;z,w)
	\right)
	\frac{dx}{x}\frac{dy}{y} \notag \\
	=
	&\notag q^{-\alpha|\mu|} t^{\beta|I|}
	 \int_{C_\mu}
	 \bigg(
	  x^{\alpha} y^{\beta}
	  \left(
	T_{q, x}^{-\mu} T_{t, y}^{I} \varphi (x, y)
	\right)
	\rho(x,y) B_{\mu, I}(t x^{-1}, q^{-1} y^{-1})\\
	&\quad\quad\quad\quad\quad\quad\quad\quad\quad\quad\quad\quad\quad\quad \times \Psi_{n,m;N,M}(x,y;z,w) 
	\bigg)
	\frac{dx}{x}\frac{dy}{y}, 
	\end{align}
where  $C_{\mu} = \mathbb{T}^n_{R} \times \mathbb{T}^m_{r}$ and $\alpha, \beta \in \mathbb{Z}$.
\end{lemm}

\begin{proof}
By the change of variables $x \to q^{- \mu}x, \ y \to t^{I}y$, we have
	\begin{align}
	&\int_{C_\mu} x^{\alpha} y^{\beta} \varphi (x, y) \rho(x,y) B_{\mu, I}(x, y) 
	\left(
	T_{q, x}^{\mu} T_{t, y}^{-I} \Psi_{n,m;N,M}(x,y;z,w)
	\right)
	\frac{dx}{x}\frac{dy}{y} \notag \\
	=
	&\notag q^{-\alpha|\mu|} t^{\beta|I|} \int_{q^{\mu}\mathbb{T}^n_{R} \times t^{-I}\mathbb{T}^m_{r}} 
	\bigg(
	x^{\alpha} y^{\beta}
	\left(
	T_{q, x}^{-\mu} T_{t, y}^{I} \varphi (x, y) \rho(x,y) B_{\mu, I}(x, y) 
	\right) \\
	&\quad\quad\quad\quad\quad\quad\quad\quad\quad\quad\quad\quad\quad\quad \times
	 \Psi_{n,m;N,M}(x,y;z,w)
	 \bigg)
	\frac{dx}{x}\frac{dy}{y}.
\end{align}
By Lemma \ref{Lemm adjoint of B}, we obtain
	\begin{align}\label{Lem int to int term mu I proof eq1}
	&\int_{q^{\mu}\mathbb{T}^n_{R} \times t^{-I}\mathbb{T}^m_{r}} 
	x^{\alpha} y^{\beta}
	\left(
	T_{q, x}^{-\mu} T_{t, y}^{I} \varphi (x, y) \rho(x,y) B_{\mu, I}(x, y) 
	\right)
	 \Psi_{n,m;N,M}(x,y;z,w)
	\frac{dx}{x}\frac{dy}{y}
	\notag \\
	=
	&\notag\int_{q^{\mu}\mathbb{T}^n_{R} \times t^{-I}\mathbb{T}^m_{r}} 
	\bigg(
	x^{\alpha} y^{\beta}
	\left(
	T_{q, x}^{-\mu} T_{t, y}^{I} \varphi (x, y)
	\right)
	\rho(x,y) B_{\mu, I}(t x^{-1}, q^{-1} y^{-1})\\
	&\quad\quad\quad\quad\quad\quad\quad\quad\quad\quad\quad\quad\quad\quad\quad \times
	 \Psi_{n,m;N,M}(x,y;z,w)
	 \bigg)
	\frac{dx}{x}\frac{dy}{y}.
\end{align}
The poles of the integrand \eqref{Lem int to int term mu I proof eq1} on the right-hand side for $(x, y)  \in (\mathbb{C}\setminus \{0\})^{n+m}$ are as follows:
\begin{align}
&x_{i}-q^{-k}t/z_j = 0 \quad   (k \in \mathbb{Z}_{ \geq 0}, \, 1\leq i \leq n, \, 1\leq j \leq N), \\
&y_{i}-t^{-k-1}/w_j = 0 \quad  (k \in \mathbb{Z}_{ \geq 0}, \, 1\leq i \leq m, \, 1\leq j \leq M), \\
&x_{i}-t q^{k + \mu_{i}}x_j = 0 \quad   (k \in \mathbb{Z}_{ \geq 0}, \, 1\leq i \neq j \leq n), \\
&y_{i}-q t^{k}y_j = 0 \quad   (k \in \mathbb{Z}_{ \geq 0}, \, 1\leq i \neq j \leq m), \\
&\label{dizon1}x_{i}-q y_j = 0 \quad   (\ 1\leq i \leq n, \, 1\leq j \leq m), \\
&x_{i}-t y_j = 0 \quad   (\ 1\leq i \leq n, \, 1\leq j \leq m), \\
&x_{i}-q^{\mu_i +1} t y_j = 0 \quad   (\ 1\leq i \leq n, \, j \in I), \\
&\label{dizon2}x_{i}-q^{\mu_i} t y_j = 0 \quad   (\ 1\leq i \leq n, \, j \notin I).
\end{align}
From assumptions $|z_{j}|<|t|/R , \, |w_{j}|<1/r$ and $|q| < |t|$, the integrand \eqref{Lem int to int term mu I proof eq1} on the right-hand side is holomorphic on $D$.
Therefore, we obtain
\begin{align}\label{Lem int to int term mu I proof eq2}
	&\notag \int_{q^{\mu}\mathbb{T}^n_{R} \times t^{-I}\mathbb{T}^m_{r}} 
	\bigg(
	x^{\alpha} y^{\beta}
	\left(
	T_{q, x}^{-\mu} T_{t, y}^{I} \varphi (x, y)
	\right)
	\rho(x,y) B_{\mu, I}(t x^{-1}, q^{-1} y^{-1})\\
	 &\quad\quad\quad\quad\quad\quad\quad\quad\quad\quad\quad\quad\quad\quad\quad \times
	 \Psi_{n,m;N,M}(x,y;z,w)
	 \bigg)
	\frac{dx}{x}\frac{dy}{y}
	\notag \\
	=
	&\int_{\mathbb{T}^n_{R} \times \mathbb{T}^m_{r}} 
	\bigg(
	x^{\alpha} y^{\beta}
	\left(
	T_{q, x}^{-\mu} T_{t, y}^{I} \varphi (x, y)
	\right)
	\rho(x,y) B_{\mu, I}(t x^{-1}, q^{-1} y^{-1})\notag\\
	&\quad\quad\quad\quad\quad\quad\quad\quad\quad\quad\quad\quad\quad\quad\quad \times
	 \Psi_{n,m;N,M}(x,y;z,w)
	 \bigg)
	\frac{dx}{x}\frac{dy}{y}.
\end{align}
From this equation, we obtain the desired equation \eqref{int to int term mu I eq}.
\end{proof}

\begin{rem}
The integrand on the right-hand side of \eqref{Lem int to int term mu I proof eq1} is holomorphic on $D$ if we assume $R < |q|^{\max_{i} {|\mu_i|} + 1} |t| r$ instead of $r < |q|^{\max_{i} {|\mu_i|}} R$ in Lemma \ref{int to int term mu I}.  
Hence, Lemma \ref{int to int term mu I} holds with $r < |q|^{\max_{i} {|\mu_i|}} R$ replaced by $R < |q|^{\max_{i} {|\mu_i|} + 1} |t| r$.
Therefore, the condition $r < |q|^d R$ in Theorem \ref{Main thm after} below can be replaced by  $R < |q|^{d + 1} |t| r$.
\end{rem}

\begin{rem}
Because of the poles \eqref{dizon1}--\eqref{dizon2}, the radii $r$ and $R$ of path $C_\mu$ depend on $\mu$. 
When $m = 0$, the poles \eqref{dizon1}--\eqref{dizon2} do not appear.
Therefore, the path can be taken independently of $\mu$.
In particular, an integral transform for eigenvalue problem of the generating function $H_{n,0}(x;u)$ for Noumi--Sano operator can be constructed.
See Appendix \ref{Appendix} for details. 
\end{rem}
We restate the main result.
\begin{thm}\label{Main thm after}
Assume $q^{\alpha} = t^{-\beta}$, $r < |q|^d R$, and $|q| < |t|$, where $\alpha, \beta \in \mathbb{Z}$ and $r, R \in \mathbb{R}_{> 0}$. 
Suppose that a function  $\varphi(x, y)$ satisfies the following conditions: 
\begin{itemize}
\item[1.] The function $\varphi(x, y)$ is a simultaneous eigenfunction of $H^{k}_{n, m}$ $(0 \leq k \leq d)$, i.e.
\begin{align}
	H^{k}_{n,m}(t x^{-1},q^{-1}y^{-1}) \varphi(x, y) = E_k \varphi(x, y), \quad 0 \leq k \leq d.
\end{align}
\item[2.]
The function $\varphi(x, y)$ is holomorphic on $D = \{(x, y)\mid |q|^{d}R \leq |x_{i}| \leq R,\, r \leq |y_{j}| \leq |t|^{-1} r \, (1 \leq i \leq n, \, 1 \leq j \leq m) \}$.
\end{itemize}
We put $C_d = \mathbb{T}^n_{R} \times \mathbb{T}^m_{r}$, 
where $\mathbb{T}^s_{p}$ denotes the $s$-dimensional torus of radius $p$.
If $|z_{j}|<|t|/R , \, |w_{j}|<1/r$, then the function 
\begin{align}
\Omega(z, w)= \int_{C_d}  \varphi(x, y) \omega_{n,m;N,M}(x,y;z,w)  \frac{dx}{x}\frac{dy}{y}
\end{align}
satisfies 
\begin{align}
	H^{d}_{N,M}(z,w)\Omega(z, w)
	=
	\left( 
	\sum_{k+l = d}\frac{(t^{N - n} q^{m - M}; q)_k}{(q; q)_k} (t^{1-N}q^M)^k t^{\beta l}E_l
	\right) \Omega(z, w). 
	\end{align}
\end{thm}
\begin{proof}
Using formula \eqref{KID rank d}, we have
\begin{align}
& H^{d}_{N,M}(z,w)\Omega(z, w) \notag \\
&\notag =
 \sum_{k+l = d}
 \bigg(
 \frac{(t^{N - n} q^{m - M}; q)_k}{(q; q)_k} (t^{1-N}q^M)^k t^{\beta l}\\
&\quad \quad \quad \times  \int_{C_d}  \varphi(x, y) x^{\alpha} y^{\beta}\left(H_{n, m}^l(x, y) \Psi_{n,m;N,M}(x,y;z,w) \right) \rho(x, y)  \frac{dx}{x}\frac{dy}{y}
 \bigg)
.
\end{align} 
By Lemma \ref{int to int term mu I}, 
\begin{align}
&\int_{C_d}  \varphi(x, y) x^{\alpha} y^{\beta}\left(H_{n, m}^l(x, y) \Psi_{n,m;N,M}(x,y;z,w) \right) \rho(x, y)  \frac{dx}{x}\frac{dy}{y} \notag \\
&=
\sum_{|\mu|+|I|=l}\int_{C_d}  \varphi(x, y) x^{\alpha} y^{\beta} B_{\mu, I}(x, y) \left(T_{q, x}^{\mu} T_{t, y}^{-I} \Psi_{n,m;N,M}(x,y;z,w) \right) \rho(x, y)  \frac{dx}{x}\frac{dy}{y} \notag\\
&=
\sum_{|\mu|+|I|=l}q^{-\alpha|\mu|} t^{\beta|I|}
	 \int_{C_d}
	 \bigg(
	  x^{\alpha} y^{\beta}
	  \left(
	T_{q, x}^{-\mu} T_{t, y}^{I} \varphi (x, y)
	\right)
	\rho(x,y) B_{\mu, I}(t x^{-1}, q^{-1} y^{-1})\notag\\
	&\quad\quad\quad\quad\quad\quad\quad\quad\quad\quad\quad\quad\quad\quad\quad \times
	 \Psi_{n,m;N,M}(x,y;z,w) 
	  \bigg)
	\frac{dx}{x}\frac{dy}{y}. 
\end{align} 
By assumption $q^{\alpha} = t^{-\beta}$, we obtain
\begin{align}
&\sum_{|\mu|+|I|=d}q^{-\alpha|\mu|} t^{\beta|I|}
	 \int_{C_d}
	 \bigg(
	  x^{\alpha} y^{\beta}
	  \left(
	T_{q, x}^{-\mu} T_{t, y}^{I} \varphi (x, y)
	\right)
	\rho(x,y) B_{\mu, I}(t x^{-1}, q^{-1} y^{-1})\notag\\
	 &\quad\quad\quad\quad\quad\quad\quad\quad\quad\quad\quad\quad\quad\quad\quad \times
	 \Psi_{n,m;N,M}(x,y;z,w) 
	 \bigg)
	\frac{dx}{x}\frac{dy}{y} \notag \\
&=t^{\beta l}
\int_{C_d}
	  x^{\alpha} y^{\beta}
	  \left(
	H^{l}_{n,m}(t x^{-1},q^{-1}y^{-1}) \varphi (x, y)
	\right)
	\rho(x,y) 
	 \Psi_{n,m;N,M}(x,y;z,w) 
	\frac{dx}{x}\frac{dy}{y}.
\end{align} 
This completes the proof of the theorem.
\end{proof}
\section{Summary and discussion}\label{secsum}
In this paper, we gave an integral transform associated with eigenvalue problem of $H_{n, m}^{d}(x, y)$ in Theorem \ref{Main thm}.
We obtained this result by the kernel identity \eqref{KIDforDNSeq}, self-adjointness  \eqref{adjoint of B} and $q$-de Rham theory \eqref{int to int term mu I eq}.

There are many related problems.
We mention two of them. 

\begin{itemize}
\item 
In this paper, we gave an integral transform associated with eigenvalue problem of $H_{n, m}^{d}(x, y)$. 
However, we did not obtain an integral transform associated with eigenvalue problem of the generating function $H_{n, m}(x, y; u)$.
It is important to derive an integral transform associated with eigenvalue problem of $H_{n, m}(x, y; u)$. 
We note that an integral transform associated with the generating function $H_{n, 0}(x;u)$ of Noumi--Sano operators can be constructed by the same method.
See Appendix \ref{Appendix} for details. 

\item
The Ruijsenaars--Macdonald $q$-difference operators and the Noumi--Sano $q$-difference operators have elliptic analogues \cite{NS21, R87}. 
The deformed Noumi--Sano $q$-difference operators have also elliptic analogues \cite{HLNR2}.
It is important to consider an elliptic analogue of integral transform for deformed Noumi--Sano operators.
We remark that an integral transform for the elliptic Ruijsenaars operators is given in \cite{LNS}.
\end{itemize}


\appendix

\section{The case of Noumi--Sano operator}\label{Appendix}
In this appendix, we construct an integral transform for the eigenvalue problem of the Noumi--Sano operator $H_{n, 0}(x;u)$.
We use the same method as in Section \ref{Proof of Main Theorem}.

We recall notations for the case of $m=0$: 
\begin{align}
&H_{n,0}(x;u)=\sum_{\mu\in(\mathbb{Z}_{\geq0})^n}(t^{1-n} u)^{|\mu|}B_{\mu,\emptyset}(x)T_{q,x}^{\mu}, \\
&B_{\mu, \emptyset}(x)=\frac{\Delta(q^\mu x)}{\Delta(x)}\prod_{i,j=1}^n\frac{(tx_i/x_j;q)_{\mu_i}}{(qx_i/x_j;q)_{\mu_i}}, \\
&\rho(x) = \prod_{1 \leq i \neq j \leq n}(x_i/x_j; q)_{\infty}/(t x_i/x_j; q)_{\infty}, \\
&\Psi_{n,0;N,0}(x;z) = \prod_{i=1}^n\prod_{j=1}^N\frac{(x_iz_j;q)_\infty}{(x_iz_j/t;q)_\infty}, \\
&\omega_{n,0;N,0}(x;z)=x^{\alpha}\Psi_{n,0;N,0}(x;z) \rho(x).
\end{align} 
\begin{lemm}
Assume $R \in \mathbb{R}_{> 0}$ and $\alpha \in \mathbb{Z}$.
Suppose a function $\varphi(x)$ is holomorphic on $D = \{x\mid |q|^{\mu_i}R \leq |x_{i}| \leq R \ (1 \leq i \leq n) \}$. 
If $|z_{j}|<|t|/R$, then we have
	\begin{align}\label{int to int term mu I eq Appendix}
	&\int_{C} x^{\alpha}  \varphi (x) \rho(x) B_{\mu, \emptyset}(x) 
	\left(
	T_{q, x}^{\mu}  \Psi_{n,0;N,0}(x;z)
	\right)
	\frac{dx}{x} \notag \\
	=
	&q^{-\alpha|\mu|}
	 \int_{C}
	  x^{\alpha} 
	  \left(
	T_{q, x}^{-\mu} \varphi (x)
	\right)
	\rho(x) B_{\mu, \emptyset}( x^{-1})
	 \Psi_{n,0;N,0}(x;z) 
	\frac{dx}{x}, 
	\end{align}
where  $C= \mathbb{T}^n_{R} $.

\end{lemm}
This can be checked in a similar manner as the proof of Lemma \ref{int to int term mu I}.
We remark that conditions of the path are independent from $\mu$.  
Therefore, we get an integral transform for eigenvalue problem of the generating function $H_{n, 0}(x;u)$ of the Noumi--Sano operator.
\begin{thm}
Assume $\alpha \in \mathbb{Z}$ and $R \in \mathbb{R}_{> 0}$. 
Suppose that a function  $\varphi(x)$ satisfies the following conditions: 
\begin{itemize}
\item[1.] The function $\varphi(x)$ is an  eigenfunction of $H_{n, 0}(x;u)$, i.e.
\begin{align}
	H_{n, 0}(x^{-1};q^{-\alpha}u) \varphi(x) = E \varphi(x).
\end{align}
\item[2.]
The function $\varphi(x)$ is holomorphic on $D = \{x\mid |q|^{\mu_i}R \leq |x_{i}| \leq R \ (1 \leq i \leq n) \}$.
\end{itemize}
We put $C = \mathbb{T}^n_{R}$, 
where $\mathbb{T}^s_{p}$ denotes the $s$-dimensional torus of radius $p$.
If $|z_{j}|<|t|/R$, then the function 
\begin{align}
\Omega(z)= \int_{C}  \varphi(x) \omega_{n,0;N,0}(x;z)  \frac{dx}{x}
\end{align}
satisfies 
\begin{align}
	H_{n, 0}(z;u) \Omega(z)
	=
	\left( 
	E \frac{(t^{1-n} u; q)_{\infty}}{(t^{1-N} u; q)_{\infty}}
	\right) \Omega(z). 
\end{align}
\end{thm}

Due to the explicit formula \cite[equation (5.17)]{NS21}
\begin{align}
	H_{n, 0}(x;u) 1 = \frac{(q^{1-n} u; q)_{\infty}}{(t u; q)_{\infty}} 1, 
\end{align}
we obtain an eigenfunction of $H_{N,0}(z; u)$ as follows:
\begin{cor}
Assume  $\alpha \in \mathbb{Z}$ and $R \in \mathbb{R}_{> 0}$. 
If $|z_{j}|<|t|/R$, then the following equation holds: 
\begin{align}
&\notag H_{n, 0}(z;u)\int_{\mathbb{T}^n_{R}} \omega_{n,0;N,0}(x;z)  \frac{dx}{x} \\
&=
	\left( 
	\frac{(q^{1-n-\alpha} u; q)_{\infty}(t^{1-n} u; q)_{\infty}}{(t q^{-\alpha} u; q)_{\infty}(t^{1-N} u; q)_{\infty}}
	\right) \int_{\mathbb{T}^n_{R} }\omega_{n,0;N, 0}(x;z)  \frac{dx}{x}. 
\end{align}
\end{cor}
\begin{rem}
In the case of the deformed Noumi--Sano operator, the path $C_d$ in Theorem \ref{Main thm after} depends on $\mu$, so we do not discuss an integral transform associated with eigenvalue problem of the generating function $H_{n, m}(x, y; u)$.
\end{rem}

\section*{Acknowledgements}
This work is supported by JSPS KAKENHI Grant Number 26KJ0129.

Address:

Taikei Fujii

Department of Mathematics,
Ochanomizu University, Japan. 

E-mail: fujii.taikei@ocha.ac.jp

Takahiko Nobukawa

Department of Mathematical Sciences, 
Aoyama Gakuin University, Japan.

E-mail: nobukawa@math.aoyama.ac.jp
\end{document}